\documentclass[12pt, twoside, leqno]{article}

\usepackage{times}
\usepackage[T1]{fontenc}
\usepackage{mathrsfs}
\usepackage{latexsym}
\usepackage[dvips]{graphics}
\usepackage{epsfig}
\usepackage{enumitem}
\usepackage{verbatim}
\usepackage{hyperref}
\usepackage{pdflscape}

\usepackage{amsmath,amsfonts,amsthm,amssymb,amscd}
\input amssym.def
\input amssym.tex
\usepackage{color}

\newcommand\be{\begin{equation}}
\newcommand\ee{\end{equation}}
\newcommand\bea{\begin{eqnarray}}
\newcommand\eea{\end{eqnarray}}
\newcommand\bi{\begin{itemize}}
\newcommand\ei{\end{itemize}}
\newcommand\ben{\begin{enumerate}}
\newcommand\een{\end{enumerate}}

\newtheorem{thm}{Theorem}[section]

\newtheorem{lem}[thm]{Lemma}
\newtheorem{prop}[thm]{Proposition}

\newtheorem{defi}[thm]{Definition}

\newtheorem{rek}[thm]{Remark}

\newcommand{\N}{\mathbb{N}}

\numberwithin{equation}{section}

\begin{document}
	
\baselineskip=17pt

\title{On Near Perfect Numbers}

\author{Peter Cohen\\ Department of Mathematics \\ Massachusetts Institute of Technology \\ Cambridge, MA 02139, USA \\ E-mail: petercohen33@gmail.com
	\and Katherine Cordwell \\  Department of Computer Science \\ Carnegie Mellon University\\ Pittsburgh,  PA 15213 \\ E-mail: kcordwel@cs.cmu.edu
	\and Alyssa Epstein \\  Stanford Law School \\ Stanford, CA 94305 \\ E-mail: wtgalyssa@gmail.com
	\and  Chung-Hang Kwan\\Department of Mathematics \\ Columbia University in the City of New York\\ New York, NY 10027 \\E-mail: ck2854@math.columbia.edu
	\and Adam Lott \\ Department of Mathematics\\  University of California, Los Angeles\\ Los Angeles, CA 90095 \\ E-mail:  adamlott99@math.ucla.edu
	\and Steven J. Miller\\ Department of Mathematics and Statistics\\ Williams College \\ Williamstown, MA 01267\\ E-mail: Steven.Miller.MC.96@aya.yale.edu }

\date{}

\maketitle

\renewcommand{\thefootnote}{}



\renewcommand{\thefootnote}{\arabic{footnote}}
\setcounter{footnote}{0}


\begin{abstract} The study of perfect numbers (numbers which equal the sum of their proper divisors) goes back to antiquity, and is responsible for some of the oldest and most popular conjectures in number theory. We investigate a generalization introduced by Pollack and Shevelev: $k$-near-perfect numbers. These are examples to the well-known pseudoperfect numbers first defined by Sierpi\'nski, and are numbers such that the sum of all but at most $k$ of its proper divisors equals the number. We establish their asymptotic order for all integers $k\ge 4$, as well as some properties of related quantities.
\end{abstract}


\section{Introduction}

Let $\sigma(n)$ be the sum of all positive divisors of $n$. A natural number $n$ is \emph{perfect} if $\sigma(n)=2n$. Perfect numbers have played a prominent role in classical number theory for millennia. A well-known conjecture claims that there are infinitely many even, but no odd, perfect numbers. Despite the fact that these conjectures remain unproven, there has been significant progress on studying the distribution of perfect numbers \cite{Vo, HoWi, Ka, Er1}, as well as generalizations. One are the \textit{pseudoperfect numbers}, which were introduced by Sierpi\'nski \cite{Si}.  A natural number is pseudoperfect if it is a sum of some subset of its proper divisors. Erd\"{o}s and Benkoski \cite{Er2, BeEr} proved that the asymptotic density for pseudoperfect numbers, as well as that of abundant numbers that are not pseudoperfect (also called \textit{weird numbers}), exist and are positive.

Pollack and Shevelev \cite{PoSh} initiated the study of a subclass of pseudoperfect numbers called \textit{near-perfect numbers}. A natural number is  $k$-near-perfect if it is a sum of all of its proper divisors with \emph{at most} $k$ exceptions.  Restriction on the number of exceptional divisors leads to asymptotic density 0. The number of $1$-near-perfect numbers up to $x$ is \footnote{ This is a result stated in \cite{AnPoPo}. In the original paper of Pollack and Shevelev \cite{PoSh}, the upper bound was given by $x^{5/6+o(1)}$.} at most $x^{3/4+o(1)}$, and in general for $k\ge 1$ the number of $k$-near-perfect numbers up to $x$ is at most $x(\log\log x)^{k-1} / \log x$.

Our first result improves the count of $k$-near-perfect numbers.

\begin{thm}\label{beat}
For any non-negative integer $k$ and real number $x\ge 1$, 	denote by $N(k;x)$ the set of $k$-near-perfect numbers up to $x$.

For any  $k\ge 4$,  there exists a constant $x_{0}(k)>0$ such that for $x\ge x_{0}(k)$, we have
\begin{equation}\label{eqn wild}
\#N(k;x)\ \asymp_{k}\  \frac{x}{\log x} (\log\log x)^{\left\lfloor\frac{\log(k+4)}{\log 2}\right\rfloor-3}.
\end{equation}
\end{thm}

Our argument is based on a partition of the set $N(k;x)$ different from that of \cite{PoSh} and this is described in Section \ref{discuss}.  This allows us to carry out an inductive argument and  reduces the count of $\#N(k;x)$ for large integers $k$ to the determination of all $k$-near-perfect numbers for small integers $k$ with a fixed number of positive divisors (see Lemma \ref{helpful4}). When $4\le k\le 11$, this even allows precise asymptotic formulae.

\begin{thm}\label{small}
For $4\le k\le 11$, there exists a constant $c_{k}>0$ such that
\begin{equation}
\#N(k;x) \ \sim \ c_{k} \frac{x}{\log x}
\end{equation}
as $x\to\infty$.
\end{thm}

Indeed, the computation of the constant $c_{k}$  follows from Lemma \ref{helpful4} and
\begin{align*}
c_{4} \ &=\  c_{5} \ \approx \ 0.1667, \ c_{6} \ \approx \ 0.2024,\\
\ c_{7} \ &= \ c_{8} \ \approx \ 0.3913, \ c_{9} \ \approx \ 0.4968, \ c_{10} \ \approx \ 0.5709, \ c_{11} \ \approx \ 0.6274.
\end{align*}

Our last result is motivated by an open question raised in \cite{BeEr}: can $\sigma(n)/n$  be arbitrarily large when $n$ is a weird number? \footnote{A number is \emph{weird} if the sum of its proper divisors is greater than itself, but no subset of these divisors sums to the original number.} We replace `weirdness' by `exact-perfectness', where a natural number is \textit{$k$-exact-perfect} if it is a sum of all of its proper divisors with \textit{exactly} $k$ exceptions. Note the result below is conditional on there being no odd perfect numbers.

\begin{thm}\label{thm: Intersections}
Let $\epsilon\in(0,2/5)$. Denote by $E(k)$ the set of all $k$-exact-perfect numbers,  $E(k;x):=E(k)\cap[1,x]$ and $E_{\epsilon}(k;x):=\{n\le x: n\in E(k), \  \sigma(n)\ge 2n+n^{\epsilon}\}$. Let $M$ be the set of all natural numbers of the form $2q$, where $q$ is a Mersenne prime\footnote{ Mersenne primes are primes of the form $2^p-1$ for some prime $p$.}. If there are no odd perfect number, then for $k$ sufficiently large and $k\not\in M$, we have
\begin{equation}\label{eqn Erdos}
    \lim_{x\to\infty} \frac{\#E_{\epsilon}(k;x)}{\#E(k;x)}\ = \ 1.
\end{equation}
\end{thm}


\subsection{Outline} In Section \ref{prep} we introduce the necessary definitions and lemmata for our theorems.  In Section \ref{discuss}, we set the stage for proving Theorem \ref{beat} and \ref{small}. In Section  \ref{proofsmall}, \ref{proofbeat} and \ref{proofint}, we prove Theorem  \ref{small}, \ref{beat} and \ref{thm: Intersections} and respectively.

\subsection{Notations}

We use the following notations and definitions.

\begin{itemize}
    \item We write $f(x) \asymp g(x)$ if there exist positive constants $c_1, c_2$ such that $c_1g(x) < f(x) < c_2g(x)$ for all sufficiently large $x$.

    \item We write $f(x) \sim g(x)$ if $\lim_{x \to \infty} f(x)/g(x) = 1$.

    \item We write $f(x) = O(g(x))$ or $f(x) \ll g(x)$ if there exists a positive constant $C$ such that $f(x) < Cg(x)$ for all sufficiently large $x$.

    \item We write $f(x) = o(g(x))$ if $\lim_{x \to \infty} f(x)/g(x) = 0$.

    \item In all cases, subscripts indicate dependence of implied constants on other parameters.

    \item Let $x\ge y\ge 2$. Denote by $\Phi(x,y)$ the set of all $y$-smooth numbers up to $x$ and  $\Phi_{j}(x,y)\ := \ \{n\le x: n=p_{1}\cdots p_{j}m_{j}, P^{+}(m_{j})\le y <p_{j}<\cdots<p_{1}\}$.

    \item We use $p$ and $p_{i}$ to denote primes, and $P^+(n)$ to denote the largest prime factor of $n$.

    \item Denote by $\tau(n)$ the number of positive divisors of $n$.

    \item Denote by $\Omega(n)$ the number of prime divisors of $n$ counting multiplicities.

    \item Denote by $N(k)$ the set of all $k$-near-perfect numbers and $N(k;x):=N(k)\cap [1,x]$.

    \item Denote by $E(k)$ the set of all $k$-exact-perfect numbers and $E(k;x):=E(k)\cap [1,x]$.

\end{itemize}



\section{Preparations}\label{prep}

In this section, we collect the necessary lemmata for our theorems. We begin with a well-known result of Landau regarding the arithmetic function  $\Omega(n)$,  the number of prime factors of $n$ \textit{counting multiplicities}, i.e., if $n=p_{1}^{a_{1}}\cdots p_{r}^{a_{r}}$, then $\Omega(n)=a_{1}+\cdots +a_{r}$.  We let
\begin{equation}
\Omega(s;x) \ := \ \{n\le x: \Omega(n)= s\}.
\end{equation}

\begin{lem}\label{lan}
	Fix an integer $s\ge 1$. As $x\to\infty$, we have
	\begin{equation}\label{eqn landau}
	\#\Omega(s;x)\ \sim\  	\#\{n\le x: n=p_{1}\cdots p_{s}, p_{1}>\cdots>p_{s}\} \ \sim \ \frac{1}{(s-1)!}\frac{x}{\log x}(\log\log x)^{s-1}.
	\end{equation}
\end{lem}

\begin{proof}
See \cite{HaWr} Theorem 437 (Section 22.18).
\end{proof}

Next, we state an elementary estimate of the number of  $y$-smooth numbers up to $x$. \footnote{ Let $y\ge 2$. A natural number $n$ is said to be $y$-smooth if all of its prime factors are at most $y$. }

\begin{lem}\label{psmooth}
Let
\begin{equation}
    u \ := \ \frac{\log x}{\log y}
\end{equation}
and  $\Phi(x,y)$ be the set of  $y$-smooth numbers up to $x$. Then uniformly for $x\ge y \ge 2$, we have
\begin{equation}\label{eqn smoothrough}
\#\Phi(x,y)\ \ll\ x\exp(-u/2),
\end{equation}

\end{lem}

\begin{proof}
See Theorem 9.5 of  \cite{DeKLu}.	
	
\end{proof}
	
Our next lemma is a standard result from sieve theory.
	
	\begin{lem}\label{sieve}
		Suppose $A$ is a finite set of natural numbers, $P$ is a set of primes, $z>0$ and $P(z)$ is the product of primes in $P$ not greater than $z$. Let
		\begin{equation*}
		S(A,P,z) \ := \  \{n\in A: (n,P(z))=1\}
		\end{equation*}
		and
		\begin{equation*}
		A_{d} \ := \ \{a\in A: d\ | \ a\} .
		\end{equation*}
		
		Assume the following conditions.
		
		\begin{enumerate}
			\item Suppose $g$ is a multiplicative function  satisfying
			\begin{equation*}
			0\ \le\ g(p) \ < \ 1\  \text{ for } \ p\in P \ \text{ and }\ g(p)=0 \ \text{ for } \  p\not\in P,
			\end{equation*}
			and there exists constants $B>0$ and $\kappa\ge 0$ such that
			\begin{equation*}
			\prod_{y\le p\le w} (1-g(p))^{-1}\ \le\ \left(\frac{\log w}{\log y}\right)^{\kappa} \exp\left(\frac{B}{\log y}\right)
			\end{equation*}
			for $2\le y<w$.
			
			\item Let $X>0$. For any square-free number $d$ with all of its prime factors in $P$, define
			\begin{equation*}
			r_{d}\ := \ \#A_{d}-Xg(d).
			\end{equation*}
			Assume that $r_{d}$ satisfies the inequality
			\begin{equation*}
			\sum_{\substack{d|P(z)\\ d\le X^{\theta}}} |r_{d}|\ \le\ C\frac{x}{(\log x)^{\kappa}}.
			\end{equation*}
			for some $\theta>0$.
			\end{enumerate}

		Then for $2\le z\le X$, we have
		\begin{equation}
		\#S(A,P,z)\ \ll_{\kappa, \theta, C, B}\ XV(z),
		\end{equation}
		where
		\begin{equation}
		V(z)\ := \ \prod_{\substack{p\le z\\ p\in P}}(1-g(p)).
		\end{equation}
	\end{lem}
	
	\begin{proof}
	For example, see  \cite{FoHa}.
	\end{proof}

In the proof of Theorem \ref{beat}, an estimate is needed for the size of the set
\begin{equation}
    \Phi_{j}(x,y)\ := \ \{n\le x: n=p_{1}\cdots p_{j}m_{j}, P^{+}(m_{j})\le y <p_{j}<\cdots<p_{1}\},
\end{equation}
where $j\ge 1$, $x\ge y\ge 2$ and $P^{+}(m_{j})$ denotes the largest prime factor of $m_{j}$. It follows from Lemmas \ref{psmooth} and \ref{sieve}.

\begin{lem}\label{jsmooth}
Suppose $j\ge 1$, $x\ge y\ge e$, and $y\le x^{\frac{1}{4(j+1)\log\log x}}$. Then
\begin{equation}
    \#\Phi_{j}(x,y)\ \ll_{j}\ \frac{x\log y}{\log x}(\log\log x)^{j-1}.
\end{equation}
\end{lem}

\begin{proof}
Introduce the following sets:
\begin{align}
Q^{(0)}(x,y) \ &:= \ \{n\le x: n=p_{1}\cdots p_{j}m_{j}, P^{+}(m_{j})\le y<p_{j}<\cdots<p_{1}\le x^{\frac{1}{j+1}}\}\\
Q^{(j)}(x,y)\ & := \  \{n\le x: n=p_{1}\cdots p_{j}m_{j}, P^{+}(m_{j})\le y<x^{\frac{1}{j+1}}<p_{j}<\cdots<p_{1}\},
\end{align}	
and for $1\le i\le j-1$,
\begin{align}
Q^{(i)}(x,y) \ := \ &\{n\le x: n=p_{1}\cdots p_{j}m_{j}, \nonumber\\
 & P^{+}(m_{j})\le y<p_{j}<\cdots<p_{i+1}\le x^{\frac{1}{j+1}}<p_{i}<\cdots<p_{1}\}.
\end{align}

Clearly, we have
\begin{equation}\label{eqn sumphi}
\#\Phi_{j}(x,y) \ = \ \sum_{0\le i \le j} Q^{(i)}(x,y).
\end{equation}

\begin{enumerate}
	
	\item By \eqref{eqn smoothrough} of Lemma \ref{psmooth},  we estimate $\#Q^{(0)}(x,y)$ as follows:
	
	 \begin{align}
	\#Q^{(0)}(x,y )&\ =\ \sum_{y<p_{j}<\cdots<p_{1}\le x^{\frac{1}{j+1}}} \sum_{\substack{P^{+}(m_{j})\le y\\ m_{j}\le x/(p_{1}\cdots p_{j})}} 1\nonumber\\
	&\ \ll\ \sum_{y<p_{j}<\cdots<p_{1}\le x^{\frac{1}{j+1}}} \frac{x}{p_{1}\cdots p_{j}}\exp\left(-\frac{\log(x/p_{1}\cdots p_{j})}{2\log y}\right)\nonumber\\
	&\ \le\ \sum_{y<p_{j}<\cdots<p_{1}\le x^{\frac{1}{j+1}}} \frac{x}{p_{1}\cdots p_{j}} \exp\left(-\frac{1}{2(j+1)}\frac{\log x}{\log y}\right)\nonumber\\
	&\ \le\ x\exp\left(-\frac{1}{2(j+1)}\frac{\log x}{\log y}\right) \left(\sum_{p\le x^{\frac{1}{j+1}}} \frac{1}{p}\right)^{j}\nonumber\\
	&\ \ll_{j}\ x(\log\log x)^{j}\exp\left(-\frac{1}{2(j+1)}\frac{\log x}{\log y}\right) \label{Q0}
	\end{align}

We have to make sure that (\ref{Q0}) is of acceptable size. Indeed, since
	\begin{equation*}
	y \ \le \  x^{\frac{1}{4(j+1)\log\log  x}},
	\end{equation*}
	it follows that
	\begin{align*}
	\frac{\log x}{\log y} \log \log x \ &= \ \exp(\log\log x+\log\log\log x-\log \log  y)\\
	 \ &\le \ \exp\left(\frac{4(j+1)(\log y)(\log\log x)}{2(j+1)(\log y)}\right)\\
	  \ &\le \ \exp\left(\frac{\log x}{2(j+1)(\log y)}\right)
	\end{align*}
	and
	\begin{equation*}
	x(\log\log x)^{j} \exp\left(-\frac{1}{2(j+1)}\frac{\log x}{\log y}\right) \ \le \ \frac{x\log y}{\log x} (\log\log x)^{j-1}.
	\end{equation*}

	\item In order to estimate $\#Q^{(j)}(x,y)$, we apply Lemma \ref{sieve} with  $A$  being the set of all natural numbers up to $x$, $P$ being the set of primes in $(y,x^{1/(j+1)}]$, $z:=x^{1/(j+1)}$, $X:=x$ and $g(d):=1/d$.  Then $S(A,P,z)$ is the set of all natural numbers up to $x$ whose prime factors are at most $y$ or at least $x^{\frac{1}{j+1}}$. (Note that there are at most $j$ prime factors can be larger than $x^{\frac{1}{j+1}}$.)

	By Mertens' estimates, we can see that all of the assumptions of Lemma \ref{sieve} are satisfied with $\kappa=B=C=1$ and $\theta=1/2$.   Hence, we have
	\begin{equation}\label{eqn Qj}
	\# Q^{(j)}(x) \ \le \ \#S(A, P, z) \ \ll \  \frac{x\log y}{\log x}.
	\end{equation}

	\item For $1\le i \le j-1$, we estimate  $\#Q^{(i)}(x,y)$ also by using Lemma \ref{sieve}. 	For any choices of primes $p_{i+1},\ldots p_{j}$ such that $y<p_{j}<\cdots<p_{i+1}\le x^{\frac{1}{j+1}}$, we choose
	\begin{equation*}
	X\ := \ \frac{x}{p_{i+1}\cdots p_{j}},
	\end{equation*}
	$A$ being the set of all natural numbers up to $X$, $P$ being the set of primes in $(y,x^{1/(j+1)}]$, $z:=x^{1/(j+1)}$ and $g(d):=1/d$. Hence,
	\begin{align}
	\#Q^{(i)}(x,y)&\ =\ \sum_{y<p_{j}<\cdots<p_{i+1}\le x^{\frac{1}{j+1}}} \sum_{\substack{P^{+}(m_{j})\le y \\ p_{1}>\cdots> p_{i}>x^{\frac{1}{j+1}}\\ p_{1}\cdots p_{i}m_{j}\le x/(p_{i+1}\cdots p_{j})}} 1 \nonumber\\
	&\ \ll\ \sum_{y<p_{j}<\cdots<p_{i+1}\le x^{\frac{1}{j+1}}} \frac{x}{p_{i+1}\cdots p_{j}} \frac{\log y}{\log x}\nonumber\\
	&\ \le\ \frac{x\log y}{\log x} \left(\sum_{p\le x^{\frac{1}{j+1}}} \frac{1}{p}\right)^{j-i}\ll_{j} \frac{x\log y}{\log x} (\log\log x)^{j-i} \label{Qi},
	\end{align}

	\end{enumerate}

The result now follows from \eqref{eqn sumphi}, (\ref{Q0}), \eqref{eqn Qj} and  (\ref{Qi}).
\end{proof}

\begin{rek}
Since
\begin{equation}
    \{n\le x: n=p_{1}\cdots p_{j}m_{j}, \ m_{j}\le y<p_{j}<\cdots<p_{1}\}\ \subset\ \Phi_{j}(x,y),
\end{equation}
it follows from Lemma \ref{lan} that
\begin{align}
\#\Phi_{j}(x,y)&\ \ge\ \sum_{m_{j}\le y} \sum_{\substack{n_{j}\le \frac{x}{m_{j}}\\ n_{j}=p_{1}\cdots p_{j} \\ \text{{\rm for some} } p_{1}>\cdots p_{j}>y}} 1 \nonumber\\
&\ \gg\ \sum_{m_{j}\le y} \frac{x/m_{j}}{\log (x/m_{j})} \left(\log\log \frac{x}{m_{j}}\right)^{j-1}\nonumber\\
&\ \ge\ \frac{x}{\log x}\left(\log\log \frac{x}{y}\right)^{j-1} \sum_{m_{j}\le y} \frac{1}{m_{j}}\nonumber\\
&\ \gg\ \frac{x\log y}{\log x} \left(\log\log\frac{x}{y}\right)^{j-1}.
\end{align}
\end{rek}


Below we state some elementary observations about near-perfect numbers.

\begin{lem}\label{helpful1}
Prime powers cannot be $k$-near-perfect for any integer $k\ge 0$.
\end{lem}
\begin{proof}
This follows directly from the definition of near-perfect numbers and the uniqueness of $q$-ary representation. \end{proof}

\begin{lem}[Euclid-Euler]\label{eulem}
All even perfect numbers are of the form $2^{p-1}(2^p-1)$, where $p$ is a Mersenne prime, i.e., a prime $p$ such that $2^p-1$ is also a prime. 	
\end{lem}

\begin{lem}\label{noddperf}
	An odd perfect number has at least 4 distinct prime factors.
\end{lem}

In fact, it is now known that an odd perfect number must have at least 10 distinct prime factors. This is due to Nielsen \cite{Niel}. The proof of an odd perfect has at least $4$ distinct prime factors is completely elementary.

The following lemma resembles the aforementioned theorem of Euclid-Euler and it serves as a  complete classification of $1$-near-perfect numbers with two distinct prime factors. This is helpful in reducing the number of cases to be considered in Lemma \ref{helpful4}.

\begin{lem}\label{helpful3}
	A $1$-near-perfect number which is not perfect and has two distinct prime factors is of the form
	\begin{enumerate}
		\item $2^{t-1}(2^t-2^k-1)$, where $2^t-2^k-1$ is prime,
		\item $2^{2p-1}(2^p-1)$, where $p$ is a Mersenne prime.
		\item $2^{p-1}(2^p-1)^2$, where $p$ is a Mersenne prime.
		\item $40$.
	\end{enumerate}
\end{lem}

\begin{proof}
	See \cite{ReCh}.
\end{proof}

Upon carrying out the recursive process as described in Section \ref{discuss} and \ref{proofbeat}, it boils down to prove the following lemma which can be done by explicit computation.

\begin{lem}\label{helpful4}
 Let $\tau(m)$ be the number of positive divisors of the positive integer $m$.
	\begin{enumerate}
		\item \label{pri}	If $\tau(m)$ is prime, then $m$ cannot be $k$-near-perfect for any integer $k\ge 0$.
		
			\item  \label{1near10}  Suppose $\tau(m)=10$.  Then
			\begin{enumerate}
				\item if $m$ is perfect, then $m=496$.
				
				\item if $m$ is $1$-near-perfect, then $m\in\{496, 368, 464\}$.
			
				\end{enumerate}

			\item  \label{2near9} Suppose $\tau(m)=9$.  Then
			\begin{enumerate}
				\item $m$ cannot be perfect.
				
				\item if $m$ is $1$-near-perfect, then $m=196$.
				
				\item if $m$ is $2$-near-perfect, then $m\in \{196, 36\}$.
				
			\end{enumerate}

		\item  \label{3near8} Suppose $\tau(m)=8$. Then
		\begin{enumerate}
			\item $m$ cannot be perfect.
			
			\item if $m$ is $1$-near-perfect, then $m\in \{24, 40, 56, 88, 104\}$.
			
			\item if $m$ is $2$-near-perfect, then $m\in \{24, 40, 56, 88, 104, 30, 54, 66\}$.
			
			\item if $m$ is $3$-near-perfect, then $m\in \{24, 40, 56, 88, 104, 30, 54, 66, 42\}$.
			
		\end{enumerate}

	\item \label{near6} Suppose $\tau(m)=6$. Then if $m$ is $k$-near-perfect for some $k\ge 0$, then  $m\in\{28, 12, 18, 20\}$.

		\item \label{near4} Suppose $\tau(m)=4$. Then  if $m$ is $k$-near-perfect for some $k\ge 0$, then $m=6$.

		\end{enumerate}
	
\end{lem}

\begin{proof}
	\
	\begin{enumerate}
		
		\item Follows immediately from  Lemma \ref{helpful1}.

		\item Suppose $m$ is a $1$-near-perfect and  $\tau(m)=10$.  Since $\tau(m)=10$, $m$ is of the form $q^9$ or $q^4 r$, where $q,r$ are distinct primes. The first case cannot happen by Lemma \ref{helpful1}.
		
		Now suppose the second case. If $m$ is perfect, by Lemma  \ref{noddperf}, it must be even. Then by Lemma \ref{eulem}, $m=q^4 r = 2^{p-1}(2^p-1)$ for some Mersenne prime $p$.   It follows that
		\begin{equation*}
		q \ = \ 2,\  p-1 \ = \ 4 \ \text{ and }\  r \ = \ 2^p-1,
		\end{equation*}
		i.e., $(q, r)=(2,31)$ and  $p=5$. Note that $q,r$ are distinct primes and $p$ is a Mersenne prime. Thus, we have $m=2^4\cdot 31 =496$.
		
		If  $m$ is $1$-near-perfect but not perfect, we use Lemma \ref{helpful3} instead and similarly, we have $m=2^4\cdot 23 = 368,\  2^4\cdot 29=464$.
		
		Thus, all the possible $m$'s are  $368, 464, 496$. \newline

		\item Suppose $m$ is a $2$-near-perfect and $\tau(m)=9$.  Since $\tau(m)=9$, $m$ is of the form $q^8$ or $q^2 r^2$, where $q,r$ are distinct primes. The first case cannot happen by Lemma \ref{helpful1}. If $m=q^2 r^2$ is  $1$-near-perfect, as before  by Lemma \ref{eulem}, \ref{noddperf} and \ref{helpful3}, the only possibility is $m \ = \ 2^2\cdot 7^2 \ = \ 196$.
		
	Now suppose $m=q^2 r^2$ is 2-near-perfect but not $1$-near-perfect. It suffices to consider the following $16$ equations by observing the symmetry of $q$ and $r$ in  $q^2 r^2$:
		
		\begin{align}
		&(1+q+q^2)(1+r+r^2)-2q^2 r^2 \nonumber\\
		\ =  \ &  1+q,\  1+q^2, \ 1+qr, \ 1+q^2 r, \ q+q^2,\  q+r,\  q^2+r, \ q+qr,  \ \ q+q^2 r, \nonumber\\
		&  r+q^2 r, \ q^2+r^2, \ q^2+q r, \ q^2+q^2 r, \ q^2+qr^2,\  qr+q^2r,\  q^2 r+qr^2 \label{92near}.
		\end{align}
		
	Given any $q, r\ge 2$, 	it is clear that  $1+q$ is the smallest among the $16$ expressions on the right side of (\ref{92near}) . We claim that if $q\ge 7$ and $r\ge 2$, then
		\begin{equation*}
		(1+q+q^2)(1+r+r^2)- 2q^2 r^2 \ < \ 1+q,
		\end{equation*}
		i.e.,
		\begin{equation*}
		f_{q}(r) \ :=\ (q^2-q-1)r^2-(1+q+q^2)r-q^2 \ >\  0.
		\end{equation*}
		This is simply a quadratic polynomial inequality in $r$. Note that $q^2-q-1>0$ and
		\begin{align*}
		\Delta(q) \ &:= \ (1+q+q^2)^2+ 4q^2(q^2-q-1) \ = \  5q^4-2q^3-q^2+2q+1 \\
		\ &> \ 0.
		\end{align*}
		Thus, if
		\begin{equation}\label{eqn range}
		r \  > \ \frac{(1+q+q^2)+\sqrt{\Delta(q)}}{2(q^2-q-1)}
		\end{equation}
		then
		\begin{equation*}
		f_{q}(r) \ > \ 0.
		\end{equation*}
		The inequality \eqref{eqn range} is satisfied with $q\ge 7$ and $r\ge 2$ since for $q\ge 7$, we have
		\begin{equation*}
		2 \ > \ \frac{(1+q+q^2)+\sqrt{\Delta(q)}}{2(q^2-q-1)}.
		\end{equation*}
		The claim follows.
		
		Thus by this claim, the left side of (\ref{92near}) is strictly less than each of the $16$  possibilities of the right side of (\ref{92near}) when $q\ge 7$ and $r\ge 2$. Now, it suffices to solve the $16$ equations in $r$ with $q=2,3,5$. The only solution is $(q,r)= (3,2)$ (i.e., $m=36$), which comes from the equation
		\begin{equation*}
		(1+q+q^2)(1+r+r^2) \ = \ 2q^2 r^2+ 1+q^2 r.
		\end{equation*}
		\newline
		
		\item Suppose $m$ is $3$-near-perfect and $\tau(m)=8$. Since $\tau(m)=8$, $m$ is of the form $q^7$, $q^3 r$ \ or\  $qrs$, where $q,r,s$ are distinct primes.  Once again by Lemma \ref{helpful1}, the first case is impossible.
		
		\begin{enumerate}
			\item Suppose $m=q^3 r$. As we have done in  (\ref{1near10}) and (\ref{2near9}), if $m$ is $1$-near-perfect, then
			\begin{align*}
			m \ = \ &2^3\cdot 7 \ = \ 56, \ 2^3\cdot 11 \ = \ 88, \ 2^3\cdot 13 \ = \ 104, \\
			 & 2^3\cdot 3 \ = \ 24, \ 2^3\cdot 5 \ = \ 40,
			\end{align*}
			by Lemma \ref{eulem}, \ref{noddperf} and \ref{helpful3}.
			
			Suppose $m=q^3 r$ is $3$-near-perfect but not $1$-near-perfect. Then it suffices to consider the  ${7\choose2}+{7\choose2} = 56$ equations formed by all of the possible pairs or triples distinct proper divisors. Following the steps in (\ref{2near9}), out of the sums of these pairs or triples, the smallest ones are $1+q$ or $1+r$.
			
			When $q\ge 5$ and $r\ge 2$, we have
			\begin{equation*}
			r  \ \ge \  2 \ > \ \frac{q^3+q^2}{q^3-q^2-q-1} \ \text{ and } \ r \ge \ 2 \ >\ \frac{q^2+q+1}{q^2-q-1}.
			\end{equation*}
			These imply that
			\begin{equation}\label{eqn ine23}
			(1+q+q^2+q^3)(1+r)- 2q^3 r \ < \ 1+q, \ 1+r.
			\end{equation}
			The same inequality is valid by replacing the right side of \eqref{eqn ine23} by the sum of any of the  $56$ possible pairs or triples of proper divisors of $m=q^3 r$, when $q\ge 5$ and $r\ge 2$.
			
			It remains to solve the $56$ equations in $r$ with $q=2, 3$. Only the following equations are solvable:
			\begin{enumerate}
				\item $(1+q+q^2+q^3)(1+r) -2q^3 r= 1+r$; $(q,r)=(2,7)$ (i.e., $m=56$),
				
				\item $(1+q+q^2+q^3)(1+r) -2q^3 r = q+q^2$; $(q,r)=(3,2)$ (i.e., $m=54$),
				
				\item $(1+q+q^2+q^3)(1+r) -2q^3 r= q+q^3$; $(q,r)=(2,5)$ (i.e., $m=40$),
				
				\item $(1+q+q^2+q^3)(1+r) -2q^3 r = q^2+q^3$; $(q,r)=(2,3)$ (i.e., $m=24$),
				
				\item $(1+q+q^2+q^3)(1+r) -2q^3 r =1+q^2+r$; $(q,r)=(2,5)$ (i.e., $m=40$), $(q,r)=(3,2)$ (i.e., $m=54$),
				
				\item $(1+q+q^2+q^3)(1+r) -2q^3 r =1+q^3+r$; $(q,r)=(2,3)$ (i.e., $m=24$),
				
				\item $(1+q+q^2+q^3)(1+r) -2q^3 r=q+q^2+rq$; $(q,r)=(2,3)$ (i.e., $m=24$)

				\end{enumerate}

			\item Suppose $m=qrs$. It cannot be perfect.   We shall use a similar strategy as above. By  symmetry, it suffices to solve the following $19$ equations one-by-one:
			\begin{align}
			&(1+q)(1+r)(1+s)- 2qrs \nonumber \\
			\ = \ & 1, \ q, \ qr, \ 1+q, \ 1+qr, \ q+r, \ q+qr, \ q+rs, \ qr+qs, \nonumber  \\
			& 1+q+r, \ 1+q+qr, \ 1+q+rs, \ 1+qr+qs, \ q+r+s, \nonumber \\
			& q+r+qr, \ q+qr+s, \ q+qr+rs, \ q+qr+qs, \ qr+rs+qs. \label{equa8qrs}
			\end{align}

			 We claim that
			\begin{equation*}
			(1+q)(1+r)(1+s)-2qrs \ < \ 1
			\end{equation*}
			for $q\ge 11$, $r\ge 5$ and $s\ge 2$.
			
			This can be verified as follows. Since $s\ge 2$, we have
			\begin{equation*}
			\frac{12s+11}{10s-2} \ < \ 5.
			\end{equation*}
			Then by $r\ge 5$, we have
			\begin{equation*}
			r \ > \ \frac{12s+11}{10s-2}.
			\end{equation*}
			This implies that
			\begin{equation*}
			11(rs-r-s-1) \ > \ rs+r+s.
			\end{equation*}
			
			By $rs-r-s-1>0$ and $q\ge 11$, we have
			\begin{equation*}
			q(rs-r-s-1)  \ > \   rs+r+s.
			\end{equation*}
			Now, the claim follows.

			Thus, it suffices to solve the equations with $q=2,3,5,7$ or  $r=2,3$. This reduces the $19$ three-variable equations in (\ref{equa8qrs}) to two-variable ones.

			Out of these equations, only the following equations are solvable:
			\begin{enumerate}
				\item $(1+q)(1+r)(1+s)-2qrs = 1+q$; $(q,r,s)=(11,2,3), (11, 3,2)$ ($m=66$),
				
				\item $(1+q)(1+r)(1+s)-2qrs = q+qr$; $(q,r,s)=(2,5,3)$ ($m=30$)
				
				\item $(1+q)(1+r)(1+s)-2qrs = 1+q+rs$; $(q,r,s)=(5,2,3), (5,3,2)$ ($m=30$)
				
				\item $(1+q)(1+r)(1+s)-2qrs = q+r+s$; $(q,r,s)=(2,3,7)$, $(2,7,3)$, $(3,2,7)$, $(3,7,2)$, $(7,2,3)$, $(7,3,2)$ (i.e., $m=42$)
				\end{enumerate}
			
		\end{enumerate}

	\item Suppose $m$ is a $k$-near-perfect number for some $k\ge 0$ and $\tau(m)=6$. Then $k\in \{0,1,2,3,4,5\}$ and  Lemma \ref{helpful1} implies that $m$ is of the form $q^2r$ with $q, r$ being distinct primes.
	
	When $k=0$, $m=28$. When $k=1$, by Lemma \ref{helpful3} we have $m\in \{12, 18,20\}$.  For $2\le k \le 5$, consider the following Diophantine equations:
	
	\begin{align*}
	\sigma(m)-2m \ = \ & (1+q+q^2)(1+r)-2q^2 r\\
	 \ \in \ &\{1+q,\  1+q^2,\  1+r, \ 1+qr,\  q+q^2,\  q+r,\  q+qr,\  q^2+r,\  q^2+qr,\\
	& r+qr, \hspace{10pt} (k=2)\\
	& 1+q+q^2,\   1+q+r,\  1+q+qr,\  1+q^2+r,\  1+q^2+qr,\  1+r+qr, \\
	&q+q^2+r,\  q+q^2+qr,\  q+r+qr,\  q^2+r+qr, \hspace{10pt} (k=3)\\
	& 1+q+q^2+r,\  1+q+q^2+qr,\ 1+q+r+qr,\  1+q^2+r+qr,\\
	& q+q^2+r+qr,  \hspace{10pt} (k=4) \\
	& 1+q+q^2+r+qr \hspace{10pt} (k=5) \}.
	\end{align*}
	
	We may express $r$ in terms of $q$ easily:
	
	\begin{align*}
	r \ = \ & 1+\frac{q+1}{q^2-q-1}, \ \frac{q}{q^2-q-1}, \ 1+\frac{2}{q-1}, \ 1+\frac{1}{q-1}, \ \frac{1}{q^2-q-1}, \ 1-\frac{q+1}{q^2+1}, \\ &1+\frac{2}{q^2-1}, \ \frac{q+1}{q^2-q}, \ \frac{1}{q-1}, 1+\frac{1+q}{q^2}  \hspace{10pt} (k=2),\\
	& 0, \ 1+\frac{1}{q-1}, \ 1+\frac{1}{q^2-1}, \ \frac{1}{q-1}, \ \frac{q}{q^2-1}, \ 1+\frac{1}{q}, \ \frac{1}{q^2-q}, \ \frac{1}{q^2-1}, \ 1+\frac{1}{q^2}, \  \frac{q+1}{q^2}, \\
	& (k=3),\\
	& 0, \ 0, \ 1, \ \frac{1}{q},  \ \frac{1}{q^2} \hspace{10pt} (k=4), \\
	& 0 \hspace{10pt} (k=5).
	\end{align*}

	The solvabilities of the equations are now apparent as only for small $q$'s the expressions are possibly integral. Also recall the restriction that $q,r$ have to be distinct primes. Thus,  only
	\begin{equation*}
	 (1+q+q^2)(1+r)-2q^2 r \ =  \  1+r
	\end{equation*}
	has solutions and $(q,r)= (2,3), (3,2)$, which correspond to $m= 12, 18$.

	\item  Suppose $m$ is a $k$-near-perfect number for some $k\ge 0$ and $\tau(m)=4$. Then $k\in \{0,1,2,3\}$ and  Lemma \ref{helpful1} implies that $m$ is of the form $qr$, where $q, r$ are distinct primes.  By also noting the symmetry of $q$ and $r$, it suffices to consider the following Diophantine equations:
	
	\begin{align*}
	\sigma(m)-2m \ &= \ (1+q)(1+r)-2qr \\
	 \ &\in \ \{0, 1, q, 1+q, q+r, 1+q+r\}.
	\end{align*}
	
	Simply expand the above equations,  we have
	\begin{align*}
	1+q+r \  = \ qr, \ 	r \ = \ q(r-1),  \ 1 \ = \ r(q-1),  \  r \ = \ 2qr, \  1 \ = \ qr, \ qr \ = \ 0
	\end{align*}
	respectively. Each of these equations are now straight-forward to solve and the only possible solution is $m=6$.

	\end{enumerate}
	
\end{proof}

The proof of Theorem \ref{thm: Intersections} rests on the study of  the equation $\sigma(n) = \ell n +k$ which is carried out by a number of authors in the past decades; for more detail, see \cite{AnPoPo,  Po1, Po2, Po3, PoPo, PoSh,  PoPoTh}.  In this article, we only need the case of \ $\ell =  2$ and adopt following definitions from the aforementioned literature.

\begin{defi}[Regular / Sporadic Solutions]\label{regular}
The solutions of \ $\sigma(n) = 2n +k$ \ of the form
\begin{equation}
    n=pm',  \text{ where }  \ p\nmid m',\  \sigma(m') \ = \ 2m',  \ \sigma(m') \ = \ k,
\end{equation}
are called \textit{regular}. All other solutions are called  \textit{sporadic}.
\end{defi}

\begin{lem}\label{spor}
Let $x\ge 3$ and $k$ be an integer. The number of sporadic solutions $n\le x$ to $\sigma(n)=2n+k$ is at most $x^{3/5+o(1)}$ as $x\to\infty$, uniformly in $k$.
\end{lem}

\begin{proof}
See \cite{PoPoTh} Theorem 4.4.
\end{proof}


\section{Outline of Theorem \ref{beat} and \ref{small}}\label{discuss}

Let us first recall the settings in \cite{PoSh}. In order to estimate the size of the set $N(k;x)$, one may partition it into the following three subsets and estimate each respectively:
\begin{align}
N_{1}(k;x) \ &:= \ \{n\in N(k;x): P^{+}(n)\le y\}, \nonumber\\
N_{2}(k;x) \ &:= \ \{n\in N(k;x) : P^{+}(n)>y \text{ and } P^{+}(n)^2|n \},  \nonumber\\
N_{3}(k;x) \ &:= \ \{n\in N(k;x): P^{+}(n)>y  \text{ and } P^{+}(n)\mid\mid n \},
\end{align}
where we shall remark on the choice of $y=y(x)$ at the end of this section.

In \cite{PoSh} they further partitioned $N_{3}(k;x)$ according to whether $\tau(m)$ is at most $k$ or not. They bounded the contribution from $\tau(m)\le k$ simply by $\frac{x}{\log x}(\log \log x)^{k-1}$, i.e., Lemma \ref{lan}. Instead, if one considers the normal order of $\log \tau(n)$, which is $(\log 2)\log \log n$, one obtains the bound $\frac{x}{\log x}(\log\log x)^{\lfloor \frac{\log k}{\log 2}\rfloor}$ for $N(k;x)$. More work is needed, though, as this is still not the correct order for $\#N(k;x)$; we thus have to partition $N_{3}(k;x)$ more carefully. This is explained as follows.

\begin{defi}\label{asso}
	Suppose  $n=pm$ with $p> P^{+}(m)$. For $k$-near-perfect number $n$, there exists a set of proper divisors $D_{n}$ of $n$ with $\#D_{n}\le k$ such that
	\begin{equation}\label{eqn sigmasum}
	\sigma(n) \ = \ 2n+\sum_{d\in D_{n}}d.
	\end{equation}
	We define the following associated sets:
	\begin{align}
	D_{n}^{(1)} \ &:= \ \{d\in D_{n}: p\nmid d \}, \nonumber \\
	D_{n}^{(2)} \ &:= \ \{d/p: d\in D_{n},\  p\mid d\}.  \label{dnsets}
	\end{align}
\end{defi}

 It is clear that $D_{n}^{(1)}$ and $D_{n}^{(2)}$ consists of positive divisors and proper divisors of $m$ respectively.

\begin{prop}\label{mnovel}
With the same settings in Definition \ref{asso}, $n$ is $k$-near-perfect and the set  $D_{n}^{(1)}$ consists of all positive divisors of $m$ if and only if  $\tau(m)\le k$ and $m\in N(k-\tau(m))$.
\end{prop}

\begin{proof}
$(\Rightarrow)$: 	Immediately from the assumptions,
\begin{align}
(1+p)\sigma(m) \ = \ \sigma(pm) \ &= \ 2pm+\sum_{d\in D_{n}^{(1)}}d+p\sum_{d\in D_{n}^{(2)}}d \nonumber\\
\ &= \ 2pm+ \sigma(m) + p\sum_{d\in D_{n}^{(2)}}d.  \label{sigma}
\end{align}
This implies
\begin{equation}\label{eqn tau}
\sigma(m)\ =\ 2m+\sum_{d\in D_{n}^{(2)}} d.\
\end{equation}
Since $\# D_{n}^{(1)}=\tau(m)$ and $\# D_{n}^{(1)}+\# D_{n}^{(2)}=\#D_{n}\le k$, we have $\tau(m)\le k$, $\# D_{n}^{(2)}\le k-\tau (m)$ and $m\in N(k-\tau(m))$.

$(\Leftarrow)$: There exists a set of proper divisors $D_{m}$ of $m$ with $\# D_{m}\le k-\tau(m)$ such that
\begin{equation*}
\sigma(m) \ = \  2m+ \sum_{d\in D_{m}} d.
\end{equation*}
Then
\begin{align}
\sigma(n) \ = \  (1+p) \sigma(m) \ = \ \sum_{d\mid m} d + p\left(2m+ \sum_{d\in D_{m}} d\right) \ = \ 2n + \sum_{d\mid m} d +p \sum_{d\in D_{m}} d.
\end{align}
Now,
\begin{equation*}
\left\{d\mid m \right\} \cup \left\{pd: d\in D_{m} \right\}
\end{equation*}
is a set of proper divisors of $n$ with at most $\tau(m)+(k-\tau(m))=k$ elements. Thus, $n$ is a $k$-near-perfect number. Also, $D_{n}^{(1)}=\{d\mid m\}$ and $D_{n}^{(2)}=D_{m}$.

\end{proof}

To facilitate discussion that follows, we introduce the following notations:

\begin{defi}
\begin{align}
    N_{3}^{(1)}(k;x) \ &:= \ \{n\le x: n=pm,\  p>\max\{y, P^{+}(m)\}, \tau(m)\le k \nonumber\\ & \hspace{60pt}   \text{ and } \ m\in N(k-\tau(m))\} \label{n3sets1}\\
    N_{3}^{(2)}(k;x) \ &:= \ N_{3}(k;x) \setminus N_{3}^{(1)}(k;x) \label{n3sets}\\
    M(k) \ &:= \ \left\{  n\in N(k): n=pm, \ p>P^{+}(m), \ D_{n}^{(1)} \subsetneq \{d\mid m \}  \right\}  \\
    M(k;x) \ &:= \ M(k)\cap [1,x]
\end{align}
\end{defi}

We carry out the above partition into $N_{1}, N_{2}, N_{3}^{(1)}, N_{3}^{(2)}$ recursively in Section \ref{proofbeat}. At each step, we show that the contributions from $N_{1}, N_{2}, N_{3}^{(2)}$ are of acceptable sizes, \footnote{ We say that the size of a quantity is \textit{acceptable} if it is not greater than that of the main term, e.g.,  $\frac{x}{\log x}(\log\log x)^{\left\lfloor\frac{\log(k+4)}{\log 2}\right\rfloor-3}$ in  Theorem \ref{beat} and \ $x/\log x$ in Theorem \ref{small}.} whereas the description for $N_{3}^{(1)}$ allows us to move onto the next step in the recursive process.  After this is done, we only have to apply Lemma \ref{helpful4}, i.e., the determination of $k$-near-perfect numbers for small integers $k$ with a fixed number of positive divisors.  In this way, we improve upon the bound $\frac{x}{\log x}(\log \log x)^{\left\lfloor\frac{\log k}{\log 2}\right\rfloor}$ and establish Theorem \ref{beat}. For more detail, see Section \ref{proofbeatout}.  As a by-product, we are able to deduce the precise asymptotic formulae for $4\le k\le 11$ in Theorem \ref{small}.

The proof of Theorem \ref{small} is simpler than that of Theorem \ref{beat}. It follows quite directly from the partition as in Theorem \ref{beat} without encountering complications of the recursive process.  We shall start with its proof  and briefly recall the essential estimates done in \cite{PoSh}  in the next section (Section \ref{proofsmall}).

Finally, we would also like to make a remark on the choice of the parameter $y$.  In \cite{PoSh}, they chose $y=x^{\frac{1}{4\log\log x}}$ for their applications. However, this is not admissible in the proof of Theorem \ref{beat}. Firstly, it is clear that the choice of \cite{PoSh} does not satisfy the conditions in Lemma \ref{jsmooth} ($j$  will be chosen in terms of $k$ in Section \ref{n31} and $j$ grows with $k$).  Secondly, in order to make sure the estimate in Lemma \ref{jsmooth} is of acceptable sizes with respect to Theorem \ref{beat}, i.e.,  smaller than $\frac{x}{\log x}(\log\log x)^{\lfloor\frac{\log(k+4)}{\log 2}\rfloor-3}$, it is essential to choose $y=(\log x)^{\alpha}$ for some $\alpha>0$. Thirdly, $\alpha$ needs to be large enough so that the contribution of $N_{3}^{(2)}$ is acceptable, see  Section \ref{n32}. We shall see $\alpha=3k+10$ is good enough. We shall stick with this choice of $y$ in Section \ref{proofsmall2} and \ref{proofbeat}. For Section \ref{rem23}, however,   we must choose a different $y$ there for better estimates.


\section{Proof of Theorem \ref{small}}\label{proofsmall}

\subsection{}\label{proofsmall1}

  We shall review the argument of \cite{PoSh} in this subsection for the convenience of readers.

 The estimations for $N_{1}(k;x)$ and $N_{2}(k;x)$ are straight-forward. Indeed,
\begin{equation}\label{eqn n1est}
\#N_{1}(k;x) \ \le \ \#\Phi(x,y)
\end{equation}
and
\begin{equation}\label{eqn n2est}
\#N_{2}(k;x) \ \le \  \#\left\{n\le x: P^{+}(n)>y, \ P^{+}(n)^2\mid n   \right\}    \ \le \ \sum_{p>y} \frac{x}{p^2} \ \ll \  \frac{x}{y}.
\end{equation}

Suppose $n=pm\in N_{3}^{(2)}(k;x)$. For the counting argument below, we shall also assume that
\begin{equation}\label{eqn tauassum}
\tau(m) \ \le \  (\log x)^{3}.
\end{equation}
This is acceptable because
\begin{equation}\label{eqn tausmall}
\#\{n\le x: \tau(n)> 2(\log x)^3\}\ll \frac{x}{(\log x)^2}.
\end{equation}
This follows from $2\tau(m)=\tau(n)$ and the crude estimate
\begin{equation*}
2(\log x)^3 \cdot \#\{n\le x: \tau(n)> 2(\log x)^3\} \ \le \ \sum_{n\le x} \tau(n) \ \ll \ x\log x.
\end{equation*}

In the following,
we count the number of possible $p$'s such that $pm\in N_{3}^{(2)}(k;x)$ for each $m\le x/y$ .  Since $n=pm$ is $k$-near-perfect,
\begin{align}
(1+p)\sigma(m)   \ = \  2pm + \sum_{d\in D_{n}^{(1)}}d+p\sum_{d\in D_{n}^{(2)}}d,
\end{align}
where the sets $D_{n}^{(1)}$ and $D_{n}^{(2)}$ are defined in Definition \ref{asso}. Reducing both sides  (mod $p$) yields
\begin{equation}\label{eqn modp}
p \ \ \bigg\lvert \left( \sigma(m)-\sum_{d\in D_{n}^{(1)}}d \right).
\end{equation}

By (\ref{n3sets}) and Proposition \ref{mnovel},
\begin{equation*}
  \sigma(m)-\sum_{d\in D_{n}^{(1)}}d \ > \ 0.
\end{equation*}
Moreover,
\begin{align*}
  \sigma(m)-\sum_{d\in D_{n}^{(1)}}d  \ \le \ \sigma(m) \ll \ m\log\log m \ \le \ x\log\log x.
\end{align*}
Thus, the number of prime factors of  $\left(\sigma(m)-\sum\limits_{d\in D_{n}^{(1)}}d \right)$ is
\begin{equation}\label{eqn primediv}
  O(\log x).
\end{equation}

Since $D_{n}^{(1)}\subset \{ d\mid m \}$ and  $\#D_{n}^{(1)}\le k$,   the number of possible values for $\left( \sigma(m)-\sum\limits_{d\in D_{n}^{(1)}}d \right)$ is
\begin{equation}\label{eqn roughct}
\ \le \ (1+\tau(m))^{k} \ \ll_{k} \  (\log x)^{3k}
\end{equation}
by \eqref{eqn tauassum}.

As a result, from \eqref{eqn modp},  \eqref{eqn primediv} and \eqref{eqn roughct}, the number of possible $p$'s is
\begin{equation}
\ \ll_{k} \ (\log x)^{3k+1}.
\end{equation}
From this, we conclude that
\begin{equation}\label{eqn n32est}
\#N_{3}^{(2)}(k;x) \ \ll_{k} \ \frac{x}{y}(\log x)^{3k+1}+ \frac{x}{(\log x)^2}.
\end{equation}

\subsection{}\label{proofsmall2}

Throughout this subsection,  we take $y:=(\log x)^{3k+10}$.

By \eqref{eqn smoothrough} of Lemma \ref{psmooth}, there exists a constant $x_{1}(k)>0$ such that for any $x\ge x_{1}(k)$,
\begin{equation}
\#N_{1}(k;x) \ \ll \ x\exp\left(-\frac{1}{2(3k+10)} \frac{\log x}{\log\log x}\right) \ \le \ \frac{x}{(\log x)^2}.
\end{equation}
Immediately from \eqref{eqn n2est} and \eqref{eqn n32est},
\begin{equation}
\#N_{2}(k;x), \#N_{3}^{(2)}(k;x) \ \ll_{k} \ \frac{x}{(\log x)^2}.
\end{equation}
Thus,  the contributions from $N_{1}(k;x)$, $N_{2}(k;x)$ and $N_{3}^{(2)}(k;x)$ are acceptable.

It remains to consider $n\in N_{3}^{(1)}(k;x)$, i.e.,  $n=pm\le x$, $p>\max\{y, P^{+}(m)\}$ and $m\in N(k-\tau(m))$. When $4\le k\le 11$, there are only finitely many such $m$ and they have been completely determined in Lemma \ref{helpful4}. Thus, by the Prime Number Theorem, we have
\begin{equation*}
\# N_{3}^{(1)}(k;x) \ \sim \ \sum_{r=4}^{k} \sum_{\substack{\tau(m)=r\\ m\in N(k-r)}}  \pi(x/m) \ \sim \  \left(\sum_{r=4}^{k} \sum_{\substack{\tau(m)=r\\ m\in N(k-r)}} \frac{1}{m}\right) \frac{x}{\log x}
\end{equation*}
as $x\to\infty$. Explicitly, the constant
\begin{equation}
c_{k} \ := \ \sum_{r=4}^{k} \sum_{\substack{\tau(m)=r\\ m\in N(k-r)}} \frac{1}{m}
\end{equation}
is equal to
\begin{align*}
c_{4} \ = \ c_{5} \ = \ \frac{1}{6}, \ c_{6} \ = \ \frac{17}{84}  \ = \ \frac{1}{6}+\frac{1}{28}  \ \approx \ 0.2024,
\end{align*}
\begin{align*}
c_{7} \ = \ c_{8} \ = \ \frac{493}{1260}\   \ = \  \frac{1}{6}+\frac{1}{12}+\frac{1}{18}+\frac{1}{20}+\frac{1}{28} \ \approx \  \ 0.3913,
\end{align*}
\begin{align*}
c_{9} \ = \ &\frac{1}{6}+\frac{1}{12}+\frac{1}{18}+\frac{1}{20}+\frac{1}{28}+\frac{1}{24}+\frac{1}{40}+\frac{1}{56}+\frac{1}{88}+\frac{1}{104} \\
& \ = \ \frac{179017}{360360} \   \approx \  0.4968,  \\
c_{10}  \ = \   &\frac{1}{6}+\frac{1}{12}+\frac{1}{18}+\frac{1}{20}+\frac{1}{28}+\frac{1}{24}+\frac{1}{40}+\frac{1}{56}+\frac{1}{88}+\frac{1}{104}\\
&+ \frac{1}{30}+\frac{1}{54}+\frac{1}{66}+ \frac{1}{196}+\frac{1}{496} \\
\ = \ & \frac{267857123}{469188720} \  \approx \ 0.5709,\\
c_{11}  \ = \ &\frac{1}{6}+\frac{1}{12}+\frac{1}{18}+\frac{1}{20}+\frac{1}{28}+\frac{1}{24}+\frac{1}{40}+\frac{1}{56}+\frac{1}{88}+\frac{1}{104}\\
&+ \frac{1}{30}+\frac{1}{54}+\frac{1}{66}+ \frac{1}{196}+\frac{1}{496}+\frac{1}{42}+\frac{1}{36}+\frac{1}{368}+\frac{1}{464} \\
\ = \ & \frac{196329752441}{312948876240} \  \approx \ 0.6274.
\end{align*}

This completes the proof of Theorem \ref{small}.

\subsection{}\label{rem23}

Before we end this section,  we would like to follow-up on a remark of \cite{PoSh} (pp. 3044) where they claimed the result
	\begin{equation}\label{eqn poshest}
	\#N(k;x)\ll x\exp(-(c_{k}+o(1))\sqrt{\log x \log\log x}),
	\end{equation}
	 for $k=2,3$, where $c_{2}=\sqrt{6}/6\approx0.4082$ and $c_{3}=\sqrt{2}/4\approx0.3535$. In view of the discussion in Section \ref{discuss}, the reason for a much smaller estimates for $k=2,3$ lies in the the nonexistence of near perfect numbers of the form $p$ or $p^2$ where $p$ is a prime (see Lemma \ref{helpful1}). This implies $N_{3}^{(1)}(3;x)$ (see \ref{n3sets1}) is an empty set .
	
	 Since a complete argument for \eqref{eqn poshest} was not given in \cite{PoSh}, we supply more detail here and hope it will be helpful for the interested readers. The argument below actually shows that one can take $c_{2}$ and $c_{3}$ to be $1/\sqrt{2}\approx 0.7071$, but this improvement is not substantial.
	
	 Here, it is essential to apply a more precise count of $\#\Phi(x,y)$ than the one given in Lemma \ref{psmooth}. From Theorem 9.15 and Corollary 9.18 of  \cite{DeKLu}, we have
	 \begin{equation}\label{eqn presmcount}
	 \#\Phi(x,y) \ = \ x \exp \left( -u\log u+O(u\log\log u) \right)
	 \end{equation}
	 and this is uniform for
	 \begin{equation}\label{eqn smcond}
	 (\log x)^3 \ \le \  y \ \le \  x,
	 \end{equation}
	 where
	 \begin{equation}
	 u \ := \ \frac{\log x}{\log y}.
	 \end{equation}
	 In this subsection, we shall choose a different $y=y(x)$  from the one taken in the rest of this article (i.e., $y=(\log x)^{3k+10}$).
	
	 We modify the estimations for $N_{3}^{(2)}(k;x)$ sketched in Section \ref{proofsmall1} slightly:
	 \begin{align}
	 &\#N_{3}^{(2)}(k;x)\nonumber\\
	  \ \le \ &\sum_{m\le \frac{x}{y}} \#\left\{p\le \frac{x}{m}: p\  \bigg| \left(\sigma(m)-\sum_{d\in D_{pm}^{^{(1)}}} d \right) \ \text{ and } \  \sigma(m)-\sum_{d\in D_{pm}^{^{(1)}}} d >0\right\}	 \nonumber\\
	  \ \ll_{k}& \ \sum_{m\le \frac{x}{y}} (\log x)\tau(m)^{k} \nonumber\\
	  \ \ll \ &\frac{x}{y}(\log x)^{2^{k}} \label{n32est2}.
	 \end{align}
	For a proof of the last estimate, see \cite{MV} eq. (2.31), pp. 61. Also, compare this with \eqref{eqn n32est}. We shall see shortly it  is a better estimate in the case of $k=2,3$ with a new choice of $y$.
	
	 Therefore for $k=2,3$, by  \eqref{eqn presmcount}, \eqref{eqn n1est}, \eqref{eqn n2est},  (\ref{n32est2}) and the fact that $N_{3}^{(2)}(k;x)=\emptyset$, we have
	 \begin{align}
	 \#N(k;x) \ &= \  \#N_{1}(k;x)+  \#N_{2}(k;x)+  \#N_{3}^{(2)}(k;x) \nonumber \\
	              \ &\ll_{k} \ x \exp \left( -u\log u+O(u\log\log u) \right)+ \frac{x}{y}(\log x)^{2^{k}} \label{23est}.
	 \end{align}
	 We optimize the last estimate by setting
	 \begin{equation}\label{eqn optu}
	 u\log u \ = \ \log y -2^{k} \log\log x.
	 \end{equation}
	 A good approximation for $u$ satisfying \eqref{eqn optu} is
	 \begin{equation}
	 u^2 \ = \ 2\ \frac{\log x}{\log\log x},
	 \end{equation}
	 i.e.,
	 \begin{equation}
	 \log y \ = \ \frac{1}{\sqrt{2}} \sqrt{\log x \log\log x}.
	 \end{equation}
	 It is to see it satisfies the requirement \eqref{eqn smcond}. Plugging this into (\ref{23est}), we have
	 \begin{equation}
	 \#N(k;x) \ \ll \  x\exp\left(-\frac{1}{\sqrt{2}}\left\{1+O\left(\frac{\log\log\log x}{\log \log x}\right)\right\}\sqrt{\log x \log\log x}\right)
	 \end{equation}
	 for $k=2,3$.


\section{Proof of Theorem \ref{beat}}\label{proofbeat}

\subsection{Outline}\label{proofbeatout}

Throughout this section, we fix $y:=(\log x)^{3k+10}$ and denote by $T_{r}(x)$  the set of natural numbers in $[1,x]$ of the form $p_{1}\cdots p_{r}m_{r}$ with $p_{1}>\cdots>p_{r}>\max\{y, P^{+}(m_{r})\}$.

The estimates required in Step $0$ are sketched in Section \ref{proofsmall1}:
\begin{equation}
    \#N_{1}(k;x), \#N_{2}(k;x),   \#N_{3}^{(2)}(k;x) \ \ll_{k} \  \frac{x}{(\log x)^2},
\end{equation}
which are of acceptable sizes. Therefore, by (\ref{n3sets1}), it now suffices to consider the set
\begin{align}\label{set1}
\left\{n\in T_{1}(x):   \tau(m_{1}) \le k,  \ m_{1}\in N\left(k-\tau(m_{1}); \frac{x}{y}\right)   \right\}.
\end{align}

In Step $1$, we estimate the size of the set \ref{set1} by repeating the partition to $m_{1}$, i.e., consider
\begin{equation}\label{eqn set1}
    \left\{n\in T_{1}(x):  \tau(m_{1})\le k, \ m_{1}\in R\left(k-\tau(m_{1}); \frac{x}{y}\right)\right\}
\end{equation}
for $R=N_{1}, N_{2}, N_{3}^{(1)}, N_{3}^{(2)}$.

\begin{itemize}
\item When $R=N_{1}, N_{2}, N_{3}^{(2)}$, the sets \eqref{eqn set1} will be shown to be of acceptable sizes $O(\frac{x}{\log x} \log\log x)$.

\item When $R=N_{3}^{(1)}$, recall from (\ref{n3sets1}) that the condition $m_{1}\in N_{3}^{(1)}(k-\tau(m_{1}); x/y)$ refers to

\begin{align}
 m_{1} \ = \ p_{2}m_{2} \ \le \  \frac{x}{y}, & \hspace{10pt} p_{2} \ > \ \max\left\{y, P^{+}(m_{2})  \right\}, \nonumber \\
 \tau(m_{2}) \le k-\tau(m_{1}), &  \hspace{10pt}    m_{2} \in N(k-\tau(m_{1}-\tau(m_{2}))).
\end{align}
Thus, the set \eqref{eqn set1} is indeed equal to
\begin{align}
    \bigg\{n\in T_{2}(x):  \tau(m_{2}) \le \frac{k}{3},  \ m_{2}\in N\left(k-3\tau(m_{2}); \frac{x}{y^2}\right)\bigg\}. \label{set2}
\end{align}
\end{itemize}

In Step $2$, we estimate the size of (\ref{set2}) by repeating the partition to $m_{2}$, so on and so forth. More generally at Step $j-1$,  we arrive at the tasks of showing the sizes of the sets
\begin{align}
\left\{ n\in T_{j-1}(x):  \tau(m_{j-1}) \le \frac{k}{2^{j-1}-1}, \ m_{j-1}\in R\left(k-(2^{j-1}-1)\tau(m_{j-1});\frac{x}{y^{j-1}}\right)\right\}
\end{align}
being  $O_{j}\left( \frac{x}{\log x} (\log\log x)^{j-1}\right)$ for $R=N_{1}, N_{2}, N_{3}^{(2)}$. This will be done in Sections \ref{n1}, \ref{n2} and \ref{n32}.

The recursion ends once we hit an \textbf{$k$-admissible} integer $j_{0}$.

\begin{defi}\label{admi}
An integer $j_{0}\ge 1$ is said  to be $k$-\textbf{admissible} if
\begin{equation}\label{eqn adm}
1\ \le \ \#\left\{m_{j_{0}}\in \N: \tau(m_{j_{0}}) \le \frac{k}{2^{j_{0}}-1}, \ m_{j_{0}}\in N(k-(2^{j_{0}}-1)\tau(m_{j_{0}})) \right\} \ < \ \infty.
\end{equation}
\end{defi}

\begin{rek}
	In view of Lemma \ref{helpful4} (\ref{pri})(\ref{near4}), \eqref{eqn adm} is equivalent to
	\begin{equation}
\#\left\{m_{j_{0}}\in \N: 4\le \tau(m_{j_{0}}) \le \frac{k}{2^{j_{0}}-1}, \ m_{j_{0}}\in N(k-(2^{j_{0}}-1)\tau(m_{j_{0}})) \right\} \ < \ \infty.
	\end{equation}
	\end{rek}

It follows from Lemma \ref{lan} that
\begin{align}\label{estn31}
&\hspace{15pt} \#\bigg\{n\in T_{j_{0}-1}(x):  \tau(m_{j_{0}-1}) \le \frac{k}{2^{j_{0}-1}-1}, \nonumber\\
& \hspace{90pt}m_{j_{0}-1}\in N_{3}^{(1)}\left(k-(2^{j_{0}-1}-1)\tau(m_{j_{0}-1}); \frac{x}{y^{j_{0}-1}}\right)\bigg\}\nonumber\\
\ &\le \ \# \left\{n\in T_{j_{0}}(x):  \tau(m_{j_{0}}) \le \frac{k}{2^{j_{0}}-1}, \ m_{j_{0}}\in N(k-(2^{j_{0}}-1)\tau(m_{j_{0}}))\right\}\nonumber\\
\ &\ll_{j_{0}} \  \frac{x}{\log x}(\log\log x)^{j_{0}-1}
\end{align}
for an $k$-admissible integer $j_{0}$. By showing that
\begin{equation}
j_{0} \ = \  j_{0}(k) \ := \  \left\lfloor \frac{\log(k+4)}{\log 2} \right\rfloor - 2
\end{equation}
is $k$-admissible in Section \ref{n31}, and together with the estimates (\ref{estn1}), (\ref{estn2}), (\ref{estn32}) and (\ref{estn31}),  we have the upper bound in Theorem \ref{beat}.

The lower bound simply follows from the observation
\begin{equation*}
6p_{1}\cdots p_{s} \ =  \ p_{1}\cdots p_{s}+2p_{1}\cdots p_{s}+3p_{1}\cdots p_{s},
\end{equation*}
where $p_{1}>\cdots >p_{s}>3$ are primes. Thus, $6p_{1}\cdots p_{s}$ is a $k_{s}$-near-perfect numbers with
\begin{equation*}
k_{s} \ := \ \tau(6p_{1}\cdots p_{s})-1-3 \ =  \ 2^{s+2}-4.
\end{equation*}

Fix any integer $k\ge 4$. Take the largest integer $s\ge 1$ such that $k_{s}\le k$, i.e.,
\begin{align*}
s \ = \ \left\lfloor \frac{\log(k+4)}{\log 2} \right\rfloor - 2.
\end{align*}
Then by Lemma \ref{lan}  one has
\begin{align}
\#N(k;x)  \ &\ge \ \#\{6p_{1}\cdots p_{s} \le x: p_{1}>\cdots >p_{s}>3   \} \nonumber\\
\ &\gg_{k} \ \frac{x}{\log x} (\log \log x)^{\left\lfloor \frac{\log(k+4)}{\log 2} \right\rfloor - 3}.
\end{align}

\subsection{Estimation for $R=N_{1}$}\label{n1}

In view of our claim that
\begin{equation*}
j_{0} \ = \ \left\lfloor \frac{\log(k+4)}{\log 2} \right\rfloor - 2
\end{equation*}
is an $k$-admissible integer, we apply Lemma \ref{jsmooth} for
\begin{equation*}
2 \ \le \  j  \ \le  \ j_{0}.
\end{equation*}
In order to meet the assumptions of Lemma \ref{jsmooth}, we restrict to $x\ge x_{0}(k)$, where $x_{0}(k)>0$ is a large constant such that for $x\ge x_{0}(k)$,
\begin{equation}
(\log x)^{3k+10} \ \le \ x^{\frac{1}{4j_{0} \log\log x}}.
\end{equation}
Hence for $x\ge x_{0}(k)$ and $2\le j\le j_{0}$, we have
\begin{align}
\#\bigg\{ n\in T_{j-1}(x):  & \ \tau(m_{j-1}) \le \frac{k}{2^{j-1}-1},  \nonumber\\
&m_{j-1}\in N_{1}\left(k-(2^{j-1}-1)\tau(m_{j-1});\frac{x}{y^{j-1}}\right)\bigg\}\nonumber
\end{align}
being bounded by
\begin{equation}\label{estn1}
\#\Phi_{j-1}(x,y) \  \ll_{j}  \  \frac{x}{\log x}(\log\log x)^{j-1} \ \ll_{k} \  \frac{x}{\log x}(\log\log x)^{j_{0}-1}.
\end{equation}

\subsection{Estimation for $R=N_{2}$}\label{n2}

 From our previous analysis, for $2\le j\le j_{0}$ we have
\begin{align}\label{estn2}
 &\#\left\{ n\in T_{j-1}(x): m_{j-1}\in N_{2}\left(k-(2^{j-1}-1)\tau(m_{j-1});\frac{x}{y^{j-1}}\right)\right\}\nonumber\\
 \ \le \ &\sum_{\substack{p_{1}>\cdots>p_{j-1}>y\\ p_{1}\cdots p_{j-1}\le x}}\sum_{\substack{m_{j-1}\le x/p_{1}\cdots p_{j-1}\\ P^{+}(m_{j-1})^2|m_{j-1}\\ P^{+}(m_{j-1})>y}} 1\nonumber\\
\ \le \ &\sum_{\substack{p_{1}>\cdots>p_{j-1}>y\\ p_{1}\cdots p_{j-1}\le x}} \frac{x}{y}\frac{1}{ p_{1}\cdots p_{j-1}}\nonumber\\
 \  \le \ &\frac{x}{y}\left(\sum_{p\le x} \frac{1}{p}\right)^{j-1}\ \ll\  \frac{x}{y}(\log\log x)^{j-1} \  \ll  \  \frac{x}{(\log x)^{3k+10}} (\log\log x)^{j_{0}-1}.
\end{align}

\subsection{Estimation of $R=N_{3}^{(2)}$}\label{n32}

Recall the notations $M(k)$ and $M(k;x)$ introduced in Section \ref{discuss}. From the argument sketched in Section \ref{proofsmall1}, we have
\begin{equation}
    \#\{n\le x: n\in M(k), P^{+}(n)>y\}\ \ll_{k}\ \frac{x}{y}(\log x)^{3k+1}.
\end{equation}
Then
\begin{align}
\# M(k;x) & \ = \ \#\{n\le x: n\in M(k), P^{+}(n)\le y\}+\#\{n\le x: n\in M(k), P^{+}(n)> y\}\nonumber\\
&\ \ll_{k}\ \#\Phi(x,y)+  \frac{x}{y}(\log x)^{3k+1}\ \ll_{k}\  \frac{x}{(\log x)^2}. \label{mkest}
\end{align}
It follows from partial summation that
\begin{equation}\label{eqn part}
\sum_{n\in M(k)}\frac{1}{n}\ < \ \infty.
\end{equation}

Therefore for $2\le j\le j_{0}$, by applying  (\ref{mkest}), \eqref{eqn part} and  Lemma \ref{lan}, we have
\begin{align}\label{estn32}
    & \hspace{10pt} \#\left\{n\in T_{j-1}(x): m_{j-1}\in N_{3}^{(2)}\left(k-(2^{j-1}-1)\tau(m_{j-1});\frac{x}{y^{j-1}}\right)\right\}\nonumber\\
    \ &\le \ \sum_{\substack{p_{1}>\cdots>p_{j-1}>y\\ p_{1}\cdots p_{j-1}\le \sqrt{x}}} \sum_{\substack{m_{j-1}\le x/p_{1}\cdots p_{j-1}\\ m_{j-1}\in M(k)}} 1+\sum_{\substack{m_{j-1}\le\sqrt{x}\\ m_{j-1}\in M(k)}}\sum_{\substack{p_{1}>\cdots>p_{j-1}>y\\ p_{1}\cdots p_{j-1}\le x/m_{j-1}}} 1\nonumber\\
    \ &\ll_{k} \ \sum_{\substack{p_{1}>\cdots>p_{j-1}>\\ p_{1}\cdots p_{j-1}\le \sqrt{x}}} \frac{\frac{x}{p_{1}\cdots p_{j-1}}}{(\log\frac{x}{p_{1}\cdots p_{j-1}})^2}+\sum_{\substack{m_{j-1}\le\sqrt{x}\\ m_{j-1}\in M(k)}} \frac{\frac{x}{m_{j-1}}}{\log \frac{x}{m_{j-1}}}\left(\log\log \frac{x}{m_{j-1}}\right)^{j-2}\nonumber\\
    \ &\ll  \ \frac{x}{(\log x)^2}\sum_{\substack{p_{1}>\cdots>p_{j-1}>y\\ p_{1}\cdots p_{j-1}\le \sqrt{x}}} \frac{1}{p_{1}\cdots p_{j-1}}+\frac{x}{\log x}(\log\log x)^{j-2}\sum_{\substack{m_{j-1}\le\sqrt{x}\\ m_{j-1}\in M(k)}} \frac{1}{m_{j-1}} \nonumber\\
    \ &\ll \  \frac{x}{(\log x)^2}\left(\sum_{p\le\sqrt{x}}\frac{1}{p}\right)^{j-1}+\frac{x}{\log x}(\log\log x)^{j-2}\nonumber\\
   \ & \ \ll  \ \frac{x}{\log x}(\log\log x)^{j_{0}-2}.
\end{align}

\subsection{Analyzing $R=N_{3}^{(1)}$}\label{n31}

We consider the following cases:

\begin{enumerate}
	\item \label{ca1}
	\begin{equation}\label{eqn range1}
	4\cdot (2^{s}-1) \ \le \  k \ < \  8\cdot ( 2^{s}-1)
	\end{equation}
	for some $s> 1$,

	\item \label{ca2}
	\begin{equation}\label{eqn disc}
	k \ = \ 8\cdot 2^{s}-\ell
	\end{equation}
	with either
	
	\begin{enumerate}
		\item $s\ge 3$ and $\ell \in \{5,6,7,8\}$; or
		
		\item $s=2$ and $\ell\in \{6,7,8\}$,
		
		\end{enumerate}

	\item \label{ca3}
	\begin{equation}
	k \ = \  27
	\end{equation}
	 (i.e., $s=2$ and $\ell=5$ in \eqref{eqn disc}).
	\end{enumerate}
	It is clear that the above covers all integers $k\ge 12$. \footnote{ Note that the cases $4\le k\le 11$ have been settled in Theorem \ref{small}. } In any case, we have
	\begin{equation}
	s \ = \  \left\lfloor\frac{\log\left(\frac{k}{4}+1\right)}{\log 2}\right\rfloor \ = \  \left\lfloor\frac{\log(k+4)}{\log 2}\right\rfloor-2.
	\end{equation}
	
	Indeed for  Case (\ref{ca1}),
	\begin{equation*}
	\frac{\log(\frac{k}{8}+1)}{\log 2} \ < \ s \ \le \ \frac{\log\left(\frac{k}{4}+1\right)}{\log 2}
	\end{equation*}
	and
	\begin{equation*}
	0 \ < \ \frac{\log\left(\frac{k}{4}+1\right)}{\log 2}- \frac{\log\left(\frac{k}{8}+1\right)}{\log 2} \ < \ 1.
	\end{equation*}
	Hence,
	\begin{equation}
	s \ = \  \left\lfloor\frac{\log\left(\frac{k}{4}+1\right)}{\log 2}\right\rfloor \ = \  \left\lfloor\frac{\log(k+4)}{\log 2}\right\rfloor-2.
	\end{equation}
	
	For Cases (\ref{ca2}) and (\ref{ca3}), we have
	\begin{align*}
	s+2 \ < \ \frac{\log (k+4)}{\log 2} \ < \ s+3,
	\end{align*}
	and
	\begin{align}
	s  \ =  \ \left\lfloor \frac{\log (k+4)}{\log 2} \right\rfloor -2.
	\end{align}

	Thus, the upper bound in Theorem \ref{beat} would follow if we establish the claim that $s$ is an $k$-admissible integer in each case.

\begin{enumerate}

\item Suppose
 \begin{equation}
 4\cdot (2^{s}-1) \ \le \  k \ < \  8\cdot ( 2^{s}-1)
\end{equation}
for some $s> 1$.   From Lemma \ref{helpful4}(\ref{near6})(\ref{near4}), the facts that
\begin{equation*}
4\ \le\ \tau(m_{s})\ \le\ \frac{k}{2^s-1} \ < \ 8
\end{equation*}
and  $m_{s}$ being near-perfect, we have  $m_{s}\in \{6,12,28,20,28  \}$. Thus,  $s$ is an $k$-admissible integer.
\newline


\item Suppose
\begin{equation}
k \ = \ 8\cdot 2^{s}-\ell
\end{equation}
with either

\begin{enumerate}
	\item $s\ge 3$ and $\ell \in \{5,6,7,8\}$; or
	
	\item $s=2$ and $\ell\in \{6,7,8\}$.
	
\end{enumerate}
In both cases, we have
\begin{equation*}
4 \ \le \ \tau(m_{s}) \ \le \ \frac{k}{2^s-1} \ < \ 9.
\end{equation*}

If $4\le \tau(m_{s})< 8$, then $m_{s}\in \{6,12,28,20,28 \}$ as in the previous case. Now suppose $\tau(m_{s})=8$. Then
\begin{align*}
m_{s} \ &\in \ N(k-(2^{s}-1)\tau(m_{s})) \\
 \ &=  \ N((2^s-1)(8-\tau(m_{s}))+(8-\ell)) \\
 \ &= \ N(8-\ell) \ \subset \ N(3).
\end{align*}
By Lemma \ref{helpful4} (\ref{3near8}), $m_{s}\in \{24, 40, 56, 88, 104, 30, 54, 66, 42\}$. As a result, $s$ is an $k$-admissible integer.
\newline


\item Suppose $k=27$. Then
\begin{align*}
4 \ \le \ \tau(m_{2}) \ \le  \ 9.
\end{align*}
If $4\le \tau(m_{2})\le 8$, then $m_{2}\in \{6,12,28,20,28 \} \cup \{24, 40, 56, 88, 104, 30, 54, 66, 42\}$ as in the previous cases. When $\tau(m_{2})=9$, $m_{2} \in   N(0)$
and there is no such $m_{2}$ by Lemma \ref{helpful4} (\ref{2near9}). Hence,  $2$ is $27$-admissible.
\end{enumerate}

This completes the proof of
\begin{align}
N(k;x) \ \ll_{k} \ \frac{x}{\log x} (\log \log x)^{\left\lfloor \frac{\log (k+4)}{\log 2} \right\rfloor-3}
\end{align}
for any $k\ge 4$ and hence the proof of Theorem \ref{beat}.


\section{Proof of Theorem \ref{thm: Intersections}}\label{proofint}

Let $\epsilon \in (0, 2/5)$. By Lemma \ref{spor},
\begin{align}
\#(E(k;x)\setminus E_{\epsilon}(k;x))\ &\le \ \#\{n\le x: n\in E(k), n=pm', p\nmid m', \sigma(m')=2m'\}\nonumber\\
& \hspace{15pt} +O(x^{3/5+\epsilon+o(1)}).
\end{align}
For $n\in E(k)$ with $n=pm'$, $p\nmid m'$ and $\sigma(m')=2m'$, we have
\begin{equation}\label{eqn sum}
pm'\ = \ \sum_{d_{1}\in D_{1}} d_{1} +p \sum_{d_{2}\in D_{2}} d_{2},
\end{equation}
where $D_{1}$ is a subset of positive divisors of $m'$, $D_{2}$ is a subset of proper divisors of $m'$ with $\# D_{1}+\# D_{2}=\tau(pm')-1-k=2\tau(m')-1-k$.\\

Suppose that $D_{1}\neq \emptyset$. Then
\begin{equation}
1\ \le\ \sum_{d_{1}\in D_{1}} d_{1}\ \le\ \sigma(m')\ = \ 2m'.
\end{equation}

Reducing \eqref{eqn sum} modulo $p$, we have
\begin{equation}
p \ \ \bigg| \ \sum_{d_{1}\in D_{1}} d_{1}.
\end{equation}

The number of possible values for $p$ is $O(\log 2m')=O(\log x)$. Thus the number of possible values for such $n$ is $O(x^{o(1)}\log x)$ by the Hornfeck-Wirsing Theorem (\cite{HoWi}), which is acceptable.

Now suppose that $D_{1}=\emptyset$. Then $\#D_{2}=2\tau(m')-1-k$ and
\begin{equation}
m'\ = \ \sum_{d_{2}\in D_{2}} d_{2}.
\end{equation}
Since $\sigma(m')=2m'$, we have $\#D_{2}=\tau(m')-1$. Therefore, $\tau(m')-1=2\tau(m')-1-k$, i.e., $\tau(m')=k$.

By the hypothesis of non-existence of odd perfect number and the Euclid-Euler Theorem, we have $m'=2^{q'-1}(2^{q'}-1)$ for some Mersenne prime $q'$. So $k=\tau(m')=2q'\in M$. Hence if $k\not\in M$, then we have a contradiction and
\begin{equation}
\#(E(k;x)\setminus E_{\epsilon}(k;x)) \ = \ O(x^{o(1)}\log x)+O(x^{3/5+\epsilon+o(1)}) \ = \ O(x^{3/5+\epsilon+o(1)}).
\end{equation}

It was shown in \cite{PoSh}, by using a form of  the Prime Number Theorem of Drmota, Mauduit and Rivat,  that for all large $k$ the number of $k$-exactly-perfect numbers up to $x$ is $\gg_{k} x/\log x$. Therefore
\begin{equation}
\frac{\#(E(k;x)\setminus E_{\epsilon}(k;x))}{\#E(k;x)}\ \ll_{k}\ \frac{\log x}{x^{2/5-\epsilon-o(1)}}
\end{equation}
and
\begin{equation}\label{eqn limit}
\lim_{x\to\infty} \frac{\#E_{\epsilon}(k;x)}{\#E(k;x)} \ = \ 1.
\end{equation}

\begin{rek}
Suppose $k\in M$. Then $k=2q$ for some Mersenne prime $q$. Let $m=2^{q-1}(2^{q}-1)$. Then $q'=q$ and so $m'=m$ in the above argument. By the Prime Number Theorem,
\begin{equation}
\limsup_{x\to\infty} \frac{\#(E(k;x)\setminus E_{\epsilon}(k;x))}{x/\log x}\ \le \ \frac{1}{m}.
\end{equation}

On the other hand, since $m$ is perfect, the number of proper divisors of $m$ is $\tau(m)-1=2q-1$. Hence $pm$ is a sum of $2q-1$ of its proper divisors. The number of proper divisors of $pm$ is $\tau(pm)-1=4q-1$. So, $pm$ is a sum of all of its proper divisors with exactly $(4q-1)-(2q-1)=2q$ exceptions, i.e., $pm\in E(k)$. Clearly $\sigma(pm)-2pm<(pm)^{\epsilon}$ if $p>(2m^{1-\epsilon})^{1/\epsilon}$ and $p\nmid m$. It follows that
\begin{equation}
\liminf_{x\to\infty}\frac{\#(E(k;x)\setminus E_{\epsilon}(k;x))}{x/\log x} \ \ge\ \frac{1}{m}.
\end{equation}

\noindent As a result,
\begin{equation}
\lim_{x\to\infty} \frac{\#(E(k;x)\setminus E_{\epsilon}(k;x))}{x/\log x}\ = \ \frac{1}{m}.
\end{equation}
\end{rek}





\section{Acknowledgments}

This work was supported in part by NSF Grants DMS1265673, DMS1347804, DMS1561945, and DMS1659037, the Williams SMALL REU Program, the Clare Boothe Luce Program, the COSINE Program of the Chinese University of Hong Kong and the Professor Charles K. Kao Research Exchange Scholarship 2015/16. We thank Kevin Ford, Charles Chun Che Li, Paul Pollack and Carl Pomerance for helpful discussions. We also thank the anonymous referee for his/her careful reading and valuable comments.



\bigskip

\end{document}